\documentclass{amsart}

\usepackage[utf8]{inputenc} 
\usepackage[T1]{fontenc} 
\usepackage{amsmath, amsthm, amssymb, booktabs, multirow, graphicx,
  booktabs, bm, float, hyperref}

\usepackage{mathtools}
\usepackage{enumitem}
\usepackage{dialogue}

\newtheorem{defin}{Definition}[section]

\newtheorem{theorem}[defin]{Theorem}

\newtheorem{lemma}[defin]{Lemma}
\newtheorem{corollary}[defin]{Corollary}
\newtheorem{definition}[defin]{Definition}

\newtheorem{example}[defin]{Example}

\newcommand{\Q}{\mathbb{Q}}
\newcommand{\R}{\mathbb{R}}

\newcommand{\Z}{\mathbb{Z}}

\newcommand{\defi}[1]{\textit{#1}}

\makeatletter
\newcommand{\bigperp}{%
  \mathop{\mathpalette\bigp@rp\relax}%
  \displaylimits
}

\newcommand{\bigp@rp}[2]{%
  \vcenter{
    \m@th\hbox{\scalebox{\ifx#1\displaystyle2.1\else1.5\fi}{$#1\perp$}}
  }%
}
\makeatother

\DeclareMathOperator{\dist}{dist}

\begin{document}

\title{The complexity of computing the covering radius of a Euclidean lattice}

\author{Frank Vallentin}
\address{F.~Vallentin, Department Mathematik/Informatik, Abteilung
  Mathematik, Universit\"at zu K\"oln, Weyertal~86--90, 50931 K\"oln,
  Germany}
\email{frank.vallentin@uni-koeln.de}

\date{September 28, 2026}

\begin{abstract}
In this note, we prove that the covering radius problem for Euclidean
lattices is complete for the second level of the polynomial hierarchy. The note also documents the
author's first experiment with generative AI as a tool for mathematical
research.
\end{abstract}

\maketitle

\markboth{F. Vallentin}{The complexity of computing the covering radius of a Euclidean lattice}

\section{Introduction}

In this note, we consider the computational complexity of computing the
covering radius of a Euclidean lattice.

\begin{definition}\label{def:covering-radius}
Let $L=B\Z^n\subseteq\R^n$ be a lattice with an invertible basis matrix
$B\in\R^{n\times n}$. The \defi{covering radius} of $L$ is
\[
  \mu(L)=\max_{x\in\R^n}\min_{v\in L}\|x-v\|_2.
\]
\end{definition}

Equivalently, $\mu(L)$ is the smallest $r\geq 0$ such that the closed
Euclidean balls of radius $r$ centered at the points of $L$ cover $\R^n$.

\begin{definition}\label{def:crp}
The \defi{covering radius problem} \textup{CRP} is the following decision
problem. An instance consists of an invertible matrix $B\in\Q^{n\times n}$
and a positive rational number $r$. The question is whether
$\mu(B\Z^n)\leq r$. It is a \textup{YES} instance if
$\mu(B\Z^n)\leq r$ and a \textup{NO} instance if $\mu(B\Z^n)>r$.
\end{definition}

The main result of this note is the following theorem.

\begin{theorem}\label{thm:main}
The covering radius problem is complete for the second level of the
polynomial hierarchy: \textup{CRP} is
$\Pi_2^{\mathsf{P}}$-complete.
\end{theorem}

Guruswami, Micciancio, and Regev~\cite{GMR} initiated the systematic study
of the computational complexity of the covering radius problem on lattices
and observed that \textup{CRP} belongs to $\Pi_2^{\mathsf{P}}$. They also
obtained algorithms and complexity upper bounds for approximation
variants. In particular, the promise problem of distinguishing
$\mu(B\Z^n)\leq r$ from $\mu(B\Z^n)>2r$ belongs to $\mathsf{AM}$.

Hardness results were previously known for other $\ell_p$ norms.
Haviv and Regev~\cite{HR} proved $\Pi_2^{\mathsf{P}}$-hardness of
approximation within a constant factor greater than one for all sufficiently
large $p$, including $p=\infty$. In the $\ell_\infty$ norm, their result
holds for every approximation factor $\gamma$ with $1\leq\gamma<3/2$.
Bennett and Ly~\cite{BL} obtained $\mathsf{NP}$-hardness of approximation
within an explicit constant factor greater than one for every fixed
$p>p_0$, where $p_0\approx35.31$, giving the first such results for explicit
finite values of $p$. In the Euclidean norm, even $\mathsf{NP}$-hardness
of the exact problem remained open, as emphasized in~\cite{BL}.

The remainder of the note is devoted to proving Theorem~\ref{thm:main}
by giving a polynomial-time reduction from quantified 3-SAT to
\textup{CRP}. In Section~\ref{sec:quantified-sat}, we define the version of quantified 3-SAT
used in the reduction and express it as a max--min quadratic optimization
problem in variables taking values in $\{-1,+1\}$. In Section~\ref{sec:lattices}, we
introduce the family of lattices used in the reduction, obtained by small
perturbations of orthogonal direct sums of copies of the root lattice $D_4$
and $\Z$, and establish their relevant geometric properties. In Section~\ref{sec:reduction},
we present the reduction and prove its correctness.

\section{Quantified 3-SAT and quadratic optimization}\label{sec:quantified-sat}

\begin{definition}\label{def:quantified-sat}
The \defi{quantified 3-SAT problem} is the following decision problem.
An instance consists of a Boolean formula
$\Phi(X,Y)=C_1\wedge\cdots\wedge C_m$ whose variables are partitioned into
two disjoint groups: the \defi{universal variables}
$X=(X_1,\ldots,X_p)$ and the \defi{existential variables}
$Y=(Y_1,\ldots,Y_q)$.
Each clause $C_j$ is a disjunction of at most three literals.
We represent true by $+1$ and false by $-1$.
The question is whether
\[
  \forall X\in\{-1,+1\}^p\quad
  \exists Y\in\{-1,+1\}^q:\quad \Phi(X,Y)\text{ is true}.
\]
It is a \textup{YES} instance if every assignment to $X$ admits an
assignment to $Y$ satisfying all clauses. It is a \textup{NO} instance
if there is an assignment to $X$ such that every assignment to $Y$
violates at least one clause.
\end{definition}

The quantified 3-SAT problem is $\Pi_2^{\mathsf{P}}$-complete~\cite{Stockmeyer}.
For our reduction, we restrict to instances with $p\geq1$ and $m\geq1$
in which every clause contains exactly three literals on pairwise distinct
variables. This restricted problem remains $\Pi_2^{\mathsf{P}}$-hard.
Indeed, the clause restrictions can be ensured by deleting tautological
clauses and repeated literals, and replacing any shorter clause $C$ by
\[
  (C\vee Y_{q+1})\wedge(C\vee\neg Y_{q+1}),
\]
where $Y_{q+1}$ is a new existential variable. We increase $q$ by one
and repeat this step as needed.
These transformations preserve the truth of the quantified formula.
Adding an unused universal variable and a clause on three new existential
variables, if necessary, ensures $p\geq1$ and $m\geq1$.

\begin{lemma}\label{lem:quadratic-encoding}
Let $\Phi(X,Y)=C_1\wedge\cdots\wedge C_m$ be a quantified 3-SAT
instance,
where each clause contains exactly three literals on distinct variables.
For each clause $C_j$, let $a_j,b_j,c_j$ be the signed variables
representing its literals.
For every clause $C_j$, introduce a new variable $Z_j$ and define
\[
  Q_\Phi(X,Y,Z)
  =\sum_{j=1}^m(a_j+b_j+c_j+Z_j-1)^2.
\]
Then for every $X\in\{-1,+1\}^p$ and $Y\in\{-1,+1\}^q$, we have
\[
  \min_{Z\in\{-1,+1\}^m}Q_\Phi(X,Y,Z)
  =m+8\,\bigl|\{j:C_j(X,Y)\text{ is false}\}\bigr|.
\]
\end{lemma}

\begin{proof}
Set $s_j=a_j+b_j+c_j$ for each clause $C_j$.
Then $s_j\in\{-3,-1,1,3\}$. Minimizing
$(s_j+Z_j-1)^2$ over $Z_j\in\{-1,+1\}$ gives $9$ when $s_j=-3$
and $1$ in the other three cases. The equality $s_j=-3$ holds
precisely when all three literals of $C_j$ are false. Since each $Z_j$
occurs in only one summand of $Q_\Phi$, the minimum over $Z$ is the
sum of these clause minima, as claimed.
\end{proof}

By Lemma~\ref{lem:quadratic-encoding}, the max--min value
\[
  \max_{X\in\{-1,+1\}^p}
  \min_{Y\in\{-1,+1\}^q}
  \min_{Z\in\{-1,+1\}^m}
  Q_\Phi(X,Y,Z)
\]
equals $m$ for a \textup{YES} instance of quantified 3-SAT and is at least
$m+8$ for a \textup{NO} instance.

\begin{example}\label{ex:quadratic-encoding}
Consider
\[
  \Phi(X,Y)
  =(X_1\vee X_2\vee Y_1)\wedge
   (X_1\vee X_2\vee\neg Y_1).
\]
The polynomial from Lemma~\ref{lem:quadratic-encoding} is
\[
  Q_\Phi=(X_1+X_2+Y_1+Z_1-1)^2
         +(X_1+X_2-Y_1+Z_2-1)^2.
\]
When $X_1=X_2=-1$, every choice of $Y_1$ violates exactly one clause;
for every other assignment to $X$, both clauses are satisfied. Thus
\[
  \max_X\min_Y\min_Z Q_\Phi(X,Y,Z)=2+8=10,
\]
and $\Phi$ is a \textup{NO} instance of quantified 3-SAT.
\end{example}

\section{Lattices and perturbations}\label{sec:lattices}

\subsection{The geometry of the root lattice {$D_4$}}

The root lattice $D_4$ consists of the integer vectors in $\R^4$ whose
coordinates have even sum:
\[
  D_4=\bigl\{z\in\Z^4:z_1+z_2+z_3+z_4\in2\Z\bigr\}.
\]

For $x\in\R^n$ and a full-rank lattice $L\subseteq\R^n$, we write
$\dist(x,L)=\min_{v\in L}\|x-v\|_2$ for the Euclidean distance from $x$ to $L$.
The \defi{Voronoi cell} of $L$ is the set of points for which
the origin is a nearest lattice point:
\[
  \{x\in\R^n:\|x\|_2\leq\|x-v\|_2\text{ for all }v\in L\}.
\]
Its translates by $L$ cover $\R^n$ and have disjoint interiors, so it
forms a fundamental domain for $\R^n/L$.

For such a lattice $L$, a \defi{deep hole} of $L$
is a point $x\in\R^n$ satisfying $\dist(x,L)=\mu(L)$.

Let $e_1,\ldots,e_4$ be the standard basis vectors. The Voronoi cell
of $D_4$ is the regular
24-cell~\cite{CS82} with vertex set $S_0\cup S_+\cup S_-$, where
\begin{align*}
  S_0&=\{\pm e_1,\ldots,\pm e_4\},\\
  S_+&=\bigl\{x/2:x\in\{-1,+1\}^4,
                    \ x_1x_2x_3x_4=+1\bigr\},\\
  S_-&=\bigl\{x/2:x\in\{-1,+1\}^4,
                    \ x_1x_2x_3x_4=-1\bigr\}.
\end{align*}
We have $|S_0|=|S_+|=|S_-|=8$.
All 24 vertices have Euclidean norm one. Hence $\mu(D_4)=1$, and
the deep holes in this Voronoi cell are precisely its vertices.
Modulo translation by the lattice $D_4$, they form three classes, represented by
\begin{equation}\label{eq:d4-deep-hole-representatives}
  c_0=e_1,\qquad
  c_+=\tfrac12(1,1,1,1),\qquad
  c_-=\tfrac12(1,1,1,-1).
\end{equation}
For $u\in\{0,+,-\}$, the eight nearest lattice points to $c_u$
form the vertex set $c_u+S_u$ of a regular cross-polytope.

\subsection{The perturbation lemma}

Let $p\geq1$ and $q\geq0$, and consider the orthogonal direct sum
\[
  L_0=D_4^p\oplus\Z^q\subseteq\R^d,
  \qquad d=4p+q.
\]
Its squared covering radius is $R:=\mu(L_0)^2=p+q/4=d/4$. Modulo
translation by the lattice $L_0$, its deep holes are represented by
\[
  h_u=(c_{u_1},\ldots,c_{u_p},1/2,\ldots,1/2),
  \qquad u=(u_1,\ldots,u_p)\in\{0,+,-\}^p,
\]
where the final $q$ coordinates of $h_u$ are $1/2$.

The vectors from $h_u$ to its nearest lattice points form the set
\[
  \mathcal S_u=S_{u_1}\times\cdots\times S_{u_p}
  \times\{-1/2,1/2\}^q.
\]
Every $v\in\mathcal S_u$ satisfies $\|v\|_2^2=R$, and both $h_u+v$ and
$h_u-v$ belong to $L_0$.

Let $T:\R^d\to\R^d$ be an invertible linear map. Its matrix $2$-norm is
$\|T\|_2=\max_{\|x\|_2=1}\|Tx\|_2$. Define
\[
  A(T)=\max_{u\in\{0,+,-\}^p}\dist(Th_u,TL_0)^2.
\]
Thus $A(T)$ is the largest squared distance to $TL_0$ attained at the
images of the deep holes of $L_0$.

\begin{lemma}\label{lem:perturb-max-min}
If $\|T-I\|_2\leq\varepsilon\leq2/(d+4)$, then
\[
  A(T)=\max_{u\in\{0,+,-\}^p}\min_{v\in\mathcal S_u}\|Tv\|_2^2.
\]
\end{lemma}

\begin{proof}
At each $h_u$, every lattice point other than those in $h_u+\mathcal S_u$
has squared distance at least $R+2$. Indeed, in each $D_4$-summand the
next squared distance after $1$ is at least $3$, and in each $\Z$-summand
the next squared distance after $1/4$ is at least $9/4$. For
$v\in\mathcal S_u$, we have $\|Tv\|_2^2\leq(1+\varepsilon)^2R$, whereas
every other displacement $w$ from $h_u$ to a point of $L_0$ satisfies
$\|Tw\|_2^2\geq(1-\varepsilon)^2(R+2)$. The difference between these
bounds is
\[
  (1-\varepsilon)^2(R+2)-(1+\varepsilon)^2R
  =2(1-\varepsilon)^2-\varepsilon d
  \geq\frac{8}{(d+4)^2}>0.
\]
Hence a nearest point to $Th_u$ in $TL_0$ is the image of one of the
original nearest points, which gives the stated formula.
\end{proof}

\begin{lemma}\label{lem:perturb-radius}
If $\|T-I\|_2\leq\varepsilon\leq1/(100d)$, then
\[
  0\leq\mu(TL_0)^2-A(T)\leq16\varepsilon^2d^2.
\]
\end{lemma}

\begin{proof}
We first estimate the distance from a point $y$ in the Voronoi cell of
$L_0$ to a deep hole. We claim that for every such $y$ there is a deep
hole $h$ of $L_0$ such that
\begin{equation}\label{eq:deep-hole-sharpness}
  \|y-h\|_2\leq4(R-\|y\|_2^2).
\end{equation}
To prove the claim, write
$y=(z_1,\ldots,z_p,t_1,\ldots,t_q)$, with each $z_k$ in the Voronoi
cell of $D_4$ and each $t_r\in[-1/2,1/2]$. For a $D_4$-summand with coordinate $z$,
write $z=\sum_i\lambda_i v_i$ as a convex combination of the vertices of
its Voronoi cell. Choose $j$ so that $\lambda_j=\max_i\lambda_i$.
All vertices have norm $1$, and distinct vertices are at least distance
$1$ apart. Thus
\[
  1-\|z\|_2^2
  =\sum_{i<k}\lambda_i\lambda_k\|v_i-v_k\|_2^2
  \geq\frac12\Bigl(1-\sum_i\lambda_i^2\Bigr)
  \geq\frac12(1-\lambda_j),
\]
since $\sum_i\lambda_i^2\leq\lambda_j\sum_i\lambda_i=\lambda_j$.
Also, $\|v_i-v_j\|_2\leq2$ for every $i$, so
\[
  \|z-v_j\|_2
  \leq\sum_{i\ne j}\lambda_i\|v_i-v_j\|_2
  \leq2(1-\lambda_j)
  \leq4(1-\|z\|_2^2).
\]
For a $\Z$-summand with coordinate $t$, choose the nearer endpoint
$s\in\{-1/2,1/2\}$. Then
\[
  |t-s|=\tfrac12-|t|
  \leq2(\tfrac14-t^2)
  \leq4(\tfrac14-t^2).
\]
Choosing these vertices and endpoints in every summand gives a deep
hole $h$ of $L_0$. The triangle inequality now yields
\[
  \|y-h\|_2
  \leq4\sum_{k=1}^p(1-\|z_k\|_2^2)
      +4\sum_{r=1}^q(\tfrac14-t_r^2)
  =4(R-\|y\|_2^2),
\]
as claimed.

By Lemma~\ref{lem:perturb-max-min}, every vector contributing to $A(T)$
has original squared norm $R$, and hence
$A(T)\geq(1-\varepsilon)^2R$. Choose a deep hole $Ty$ of $TL_0$.
By translating it by a vector of $TL_0$, we may assume that $y$ lies
in the Voronoi cell of $L_0$. Since the origin belongs to $L_0$,
\[
  (1-\varepsilon)^2R\leq A(T)\leq\mu(TL_0)^2
  \leq\|Ty\|_2^2\leq(1+\varepsilon)^2\|y\|_2^2.
\]
By~\eqref{eq:deep-hole-sharpness}, there is a deep hole $h$ of $L_0$
such that
\[
  \|T(y-h)\|_2
  \leq4(1+\varepsilon)(R-\|y\|_2^2)
  \leq\frac{4\varepsilon d}{1+\varepsilon}
  \leq4\varepsilon d.
\]
Let $u$ be the class of $h$ and choose $v\in\mathcal S_u$ minimizing
$\|Tv\|_2^2$. Since $h+v$ and $h-v$ belong to $L_0$, the
definition of distance and the parallelogram identity give
\begin{align*}
  \mu(TL_0)^2
  &=\dist(Ty,TL_0)^2\\
  &\leq\min\!\left\{\|T(y-h-v)\|_2^2,\,
                   \|T(y-h+v)\|_2^2\right\}\\
  &\leq\tfrac12\!\left(\|T(y-h-v)\|_2^2
                      +\|T(y-h+v)\|_2^2\right)\\
  &=\|Tv\|_2^2+\|T(y-h)\|_2^2
   \leq A(T)+16\varepsilon^2d^2.
\end{align*}
The reverse inequality $A(T)\leq\mu(TL_0)^2$ holds by the definition
of the covering radius.
\end{proof}

\begin{corollary}\label{cor:perturb-quadratic}
Suppose $\|T-I\|_2\leq\varepsilon\leq1/(100d)$, and set $N=T-I$.
Then
\[
  \mu(TL_0)^2
  =R+\max_{u\in\{0,+,-\}^p}\min_{v\in\mathcal S_u}2v^\top Nv+e,
  \qquad 0\leq e\leq17\varepsilon^2d^2.
\]
Thus, up to error $e$, the squared covering radius of $TL_0$
is $R$ plus the value of a max--min quadratic program.
\end{corollary}

\begin{proof}
For every $v\in\mathcal S_u$, we have $\|v\|_2^2=R$ and
\[
  \|Tv\|_2^2=R+2v^\top Nv+\|Nv\|_2^2,
  \qquad 0\leq\|Nv\|_2^2\leq\varepsilon^2R.
\]
Taking the minimum over $v$ and then the maximum over $u$, Lemma~
\ref{lem:perturb-max-min} shows that $A(T)$ differs from
$R+\max_u\min_{v\in\mathcal S_u}2v^\top Nv$ by a number between $0$
and $\varepsilon^2R$. Lemma~\ref{lem:perturb-radius} adds an error
between $0$ and $16\varepsilon^2d^2$. Since $R=d/4\leq d^2$, their sum
is at most $17\varepsilon^2d^2$.
\end{proof}

\section{Proof of the main theorem. The Reduction.}\label{sec:reduction}

Let $\Phi(X,Y)=C_1\wedge\cdots\wedge C_m$ be an instance of the
restricted quantified 3-SAT problem from Section~\ref{sec:quantified-sat},
with $p$ universal variables and $q$ existential variables. For each
assignment to $X$, we choose $Y$ to minimize the number of violated
clauses; then we take the maximum over all assignments to $X$. Denote
this number by
\[
  \eta_\Phi=
  \max_{X\in\{-1,+1\}^p}\min_{Y\in\{-1,+1\}^q}
  \bigl|\{j:C_j(X,Y)\text{ is false}\}\bigr|.
\]
Thus $\eta_\Phi=0$ for a \textup{YES} instance and $\eta_\Phi\geq1$
for a \textup{NO} instance. The reduction converts this gap in violated
clauses into a gap in covering radius.

By Lemma~\ref{lem:quadratic-encoding},
\[
  \max_X\min_Y\min_Z Q_\Phi(X,Y,Z)=m+8\eta_\Phi.
\]

We specialize the lattice from Section~\ref{sec:lattices} to
\[
  L_0=D_4^p\oplus\Z^{q+m+1}\subseteq\R^d,
  \qquad d=4p+q+m+1,
  \qquad R=\mu(L_0)^2=\frac d4.
\]
For $u\in\{0,+,-\}^p$, write a vector in $\mathcal S_u$ as
\[
  \xi=(v_1,\ldots,v_p,w_0,\ldots,w_{q+m}),
  \qquad v_i\in S_{u_i},\quad w_j\in\{-1/2,1/2\}.
\]
The coordinate $w_0$ will supply a reference sign, the coordinates
$w_1,\ldots,w_q$ will encode the existential variables, and the
coordinates $w_{q+1},\ldots,w_{q+m}$ will encode the auxiliary variables from
Lemma~\ref{lem:quadratic-encoding}.

In Section~\ref{subsec:quadratic-encoding} we use the deep-hole classes
of $L_0$ to encode the values of the universal variables and construct
a homogeneous quadratic polynomial. In Section~\ref{subsec:covering-gap}
we turn this polynomial into a small perturbation of $L_0$ and obtain
the covering-radius gap.

\subsection{A homogeneous quadratic encoding}\label{subsec:quadratic-encoding}

We construct a homogeneous quadratic polynomial whose minimum over $\mathcal S_u$
encodes the smallest number of violated clauses for the universal
assignment represented by $u$.
First, we choose one vector from each set $S_{u_i}$. Then we use the
coordinates $w_1,\ldots,w_{q+m}$, each belonging to a $\Z$-summand
of $L_0$, to represent the values of the existential and auxiliary
variables.

Set $\beta=(2,1,1,0)$ and $\ell(z)=z_1-3z_2$ for $z\in\R^4$.
The three vectors $c_u$ defined in Section~\ref{sec:lattices},
equation~\eqref{eq:d4-deep-hole-representatives}, satisfy
\[
  \beta\cdot c_u=2,
  \qquad \beta\cdot z\leq1\quad(z\in S_u\setminus\{c_u\}),
\]
and
\[
  \ell(c_0)=1,\qquad \ell(c_+)=\ell(c_-)=-1.
\]
Thus $u_i=0$ encodes $X_i=+1$ (true), while $u_i\in\{+,-\}$ encodes
$X_i=-1$ (false). The first display says that
$\beta\cdot z$ selects $c_u$ from $S_u$, with a gap of at least $1$.

Write $\xi=(v_1,\ldots,v_p,w_0,\ldots,w_{q+m})\in\mathcal S_u$.
The reference sign is $\tau=2w_0\in\{-1,+1\}$. Because $S_u$ is
centrally symmetric, for fixed $\tau$ the expression
$\tau\beta\cdot v_i$ is uniquely maximized at
$v_i=\tau c_{u_i}$. We use this observation to select the vectors
$v_i$ without fixing the value of $\tau$.

Define the following linear forms in the coordinates of $\xi$:
\begin{align*}
  t_i&=\ell(v_i)\quad(1\leq i\leq p),
  &t_{p+r}&=2w_r\quad(1\leq r\leq q),\\
  z_j&=2w_{q+j}\quad(1\leq j\leq m).
\end{align*}
For each clause $C_j$, form $a_j,b_j,c_j$ as in
Lemma~\ref{lem:quadratic-encoding}, replacing $X_i$ by $t_i$ and $Y_r$
by $t_{p+r}$. For example, the literal $\neg Y_r$ contributes
$-t_{p+r}$. Set $M=10m$ and define
\begin{equation}\label{eq:reduction-energy}
  \mathcal E(\xi)
  =-M\tau\sum_{i=1}^p\beta\cdot v_i
   +\sum_{j=1}^m(a_j+b_j+c_j+z_j-\tau)^2.
\end{equation}
The first sum selects the vectors $v_i$; the sum of squares encodes
the clauses. Thus $\mathcal E$ is a homogeneous quadratic polynomial.

\begin{example}\label{ex:reduction-energy}
Consider the formula from Example~\ref{ex:quadratic-encoding}:
\[
  \Phi(X,Y)
  =(X_1\vee X_2\vee Y_1)\wedge
   (X_1\vee X_2\vee\neg Y_1).
\]
Here $p=2$, $q=1$, $m=2$, and $M=20$. Write
$\xi=(v_1,v_2,w_0,w_1,w_2,w_3)$, so that
$\tau=2w_0$, $t_1=\ell(v_1)$, $t_2=\ell(v_2)$, $t_3=2w_1$,
$z_1=2w_2$, and $z_2=2w_3$. Equation~\eqref{eq:reduction-energy}
becomes
\begin{align*}
  \mathcal E(\xi)
  ={}&-20\tau\beta\cdot(v_1+v_2)\\
     &+(t_1+t_2+t_3+z_1-\tau)^2\\
     &+(t_1+t_2-t_3+z_2-\tau)^2.
\end{align*}
Writing $v_i=(v_{i1},v_{i2},v_{i3},v_{i4})$, this is, in the
coordinates of $\xi$,
\begin{align*}
  \mathcal E(\xi)
  ={}&-40w_0\bigl(2(v_{11}+v_{21})+(v_{12}+v_{22})+(v_{13}+v_{23})\bigr)\\
     &+\bigl(v_{11}+v_{21}-3(v_{12}+v_{22})+2w_1+2w_2-2w_0\bigr)^2\\
     &+\bigl(v_{11}+v_{21}-3(v_{12}+v_{22})-2w_1+2w_3-2w_0\bigr)^2.
\end{align*}
\end{example}

\begin{lemma}\label{lem:reduction-energy-value}
The homogeneous quadratic polynomial $\mathcal E$ defined in~\eqref{eq:reduction-energy}
satisfies
\[
  \max_{u\in\{0,+,-\}^p}\min_{\xi\in\mathcal S_u}\mathcal E(\xi)
  =-2Mp+m+8\eta_\Phi.
\]
\end{lemma}

\begin{proof}
Fix $u$. For either value of $\tau$, setting
$v_i=\tau c_{u_i}$ for every $i$ makes the first sum equal to
$-2Mp$. For any choice of the existential coordinates, we can then
choose the auxiliary coordinates so that each clause square is at most
$9$. Thus some $\xi\in\mathcal S_u$ satisfies
$\mathcal E(\xi)\leq-2Mp+9m$. If any $v_i\ne\tau c_{u_i}$, the first sum is at least
$-2Mp+M$, and all clause squares are nonnegative. Since $M=10m>9m$,
every minimizer over $\mathcal S_u$ satisfies $v_i=\tau c_{u_i}$ for all $i$.

Under these constraints, the value of the universal variable $X_i$ is
$X_i(u)=\ell(c_{u_i})=\tau t_i$. The values of the existential variables
$Y_r=\tau t_{p+r}$ and the auxiliary variables $Z_j=\tau z_j$ can be
chosen independently.
Since $\tau^2=1$, each clause square becomes
\[
  (a_j+b_j+c_j+z_j-\tau)^2
  =(\tau a_j+\tau b_j+\tau c_j+\tau z_j-1)^2,
\]
which is the $j$th summand of $Q_\Phi(X(u),Y,Z)$. The proof of
Lemma~\ref{lem:quadratic-encoding} shows that minimizing these summands
over $Z$ contributes $1$ for each satisfied clause and $9$ for each
false clause. Hence
\begin{align*}
  \min_{\xi\in\mathcal S_u}\mathcal E(\xi)
  &=-2Mp+\min_{Y,Z}Q_\Phi(X(u),Y,Z)\\
  &=-2Mp+m+8\min_Y
    \bigl|\{j:C_j(X(u),Y)\text{ is false}\}\bigr|.
\end{align*}
Every universal assignment occurs as $X(u)$, giving the claimed identity.
\end{proof}

\subsection{The perturbation and the covering-radius gap}\label{subsec:covering-gap}

Every vector $\xi\in\mathcal S_u$ has squared norm $R$. We can
therefore add a multiple of $\|\xi\|_2^2$ to $\mathcal E$ to shift all
values in Lemma~\ref{lem:reduction-energy-value} by the same amount. Let
$H_\Phi$ be the symmetric rational matrix defined by
\begin{equation}\label{eq:reduction-matrix}
  2\xi^\top H_\Phi\xi
  =\mathcal E(\xi)
   +\frac{2Mp-m-4}{R}\|\xi\|_2^2
   \qquad(\xi\in\R^d).
\end{equation}
Lemma~\ref{lem:reduction-energy-value} now gives
\begin{equation}\label{eq:reduction-quadratic-value}
  \max_{u\in\{0,+,-\}^p}\min_{\xi\in\mathcal S_u}
  2\xi^\top H_\Phi\xi=8\eta_\Phi-4.
\end{equation}
Thus the max--min quadratic value is $-4$ for a \textup{YES} instance
of quantified 3-SAT and at least $4$ for a \textup{NO} instance.

\begin{lemma}\label{lem:reduction-matrix-norm}
The matrix $H_\Phi$ defined in~\eqref{eq:reduction-matrix} satisfies
\[
  \|H_\Phi\|_2\leq3d^2.
\]
\end{lemma}

\begin{proof}
The selector term in~\eqref{eq:reduction-energy} contributes a matrix of norm
$M\sqrt{6p}/2$. Each clause square contributes a rank-one matrix
$f_jf_j^\top/2$ with $\|f_j\|_2^2\leq38$: the three literals use
distinct variables and contribute at most $10$ each, while the
reference and auxiliary coordinates contribute $4$ each. Since $R\geq p$,
the diagonal shift in~\eqref{eq:reduction-matrix} has norm at most $M$. Hence
\[
  \|H_\Phi\|_2
  \leq 5m\sqrt{6p}+29m
  \leq 3d^2.
\]
For the last inequality, use $5\sqrt{6p}\leq13p$,
$mp\leq d^2/16$, and $m\leq d^2/20$, which follow from
$p\geq1$ and $d\geq4p+m+1$.
\end{proof}

Set
\[
  \delta=\frac1{100d^6},\qquad N=\delta H_\Phi,
  \qquad T=I+N.
\]
By Lemma~\ref{lem:reduction-matrix-norm},
$\|T-I\|_2\leq3\delta d^2\leq1/(100d)$, so $T$ is invertible
and Corollary~\ref{cor:perturb-quadratic} applies. Together with
\eqref{eq:reduction-quadratic-value}, it gives
\begin{equation}\label{eq:reduction-radius-gap}
  \mu(TL_0)^2=R+\delta(8\eta_\Phi-4)+e,
  \qquad 0\leq e\leq153\delta^2d^6<2\delta.
\end{equation}
Consequently, a \textup{YES} instance satisfies
$\mu(TL_0)^2<R-2\delta$, whereas a \textup{NO} instance satisfies
$\mu(TL_0)^2\geq R+4\delta$.

Using an integer basis matrix $B_0$ of $L_0$ and a rational $r$ with
$R\leq r^2<R+\delta$, computed by bisection, the gap above gives a
polynomial-time reduction to the \textup{CRP} instance $(TB_0,r)$,
whose entries have polynomial bit length.
Finally, \textup{CRP} belongs to $\Pi_2^{\mathsf P}$~\cite{GMR}. This completes the proof of Theorem~\ref{thm:main}.

\section*{Disclosure of AI usage}

On Sunday, 20 September 2026, I spent some time discussing
complexity theory in the geometry of numbers with generative AI. In
particular, I wanted to understand
whether finding cold spots in a lattice is a computationally difficult
problem.

This problem underlies some recent work that Christine Bachoc,
Philippe Moustrou, Marc Christian Zimmermann, and I did together~\cite{BMVZ}. I had also
thought that it should be easier than the covering radius problem (CRP) for
lattices, whose exact complexity has been open for more than twenty
years.

CRP has been close to my heart since I was a student. For many years,
I carried around a worn photocopy of McLoughlin's paper on the hardness
of computing the covering radius of binary linear codes~\cite{McLoughlin}.
I tried several times to resolve its complexity, but the only related result
that Mathieu Dutour Sikiri\'c, Achill Sch\"urmann, and I managed to prove
was that computing the number of vertices of the Voronoi cell
of a lattice is $\#\mathsf{P}$-hard~\cite{DSV}. This is connected because the only
general computational procedure I know for determining the covering
radius proceeds by computing the vertices of the Voronoi cell.

Later that day, after the generative AI had worked out a hardness
proof for finding cold spots, I asked whether essentially the same
reduction could be adapted to CRP. After some back and forth, it
turned out that it could.

I have since digested the proof, and I find it quite beautiful.
After removing the AI-generated clutter and simplifying the argument,
I was quite motivated to write down the proof, both for myself and
for others.

I do not really know how much of the discovery
should be attributed to me and how much of the heavy lifting was done
by the AI. In the end, I do not care very much. I would be perfectly
comfortable being regarded mainly as the messenger of the proof. My
motivation comes from my long-standing obsession with the problem and
from the beauty of the argument itself. I also wanted, at least once,
to experience directly how much work it currently takes, in September
2026, to turn an AI-generated mathematical argument into mathematics
that I would actually be willing to publish under my name.
This took much less work than I had anticipated. The text of this note
was drafted and revised by AI under my direction on the morning of
Friday, 25 September 2026, and a final touch was added on Monday,
28 September 2026.
I take full responsibility for all the contents of this note, including
the correctness of its mathematical claims and the accuracy and completeness of its
references.

\end{document}